\documentclass[12pt]{amsart}
\usepackage{a4wide}
\usepackage[T1]{fontenc}
\usepackage{amssymb,amsmath,amsthm,latexsym}
\usepackage{mathrsfs}
\usepackage[usenames,dvipsnames]{color}
\usepackage{euscript}
\usepackage{graphicx}
\usepackage{mdwlist}
\usepackage{enumerate}
\usepackage{mathtools,dsfont,wasysym}
\usepackage{bbm}
\usepackage{stmaryrd}
\usepackage{centernot}
\usepackage{todonotes}
\usepackage{hyperref}
\usepackage{subfigure}
\hypersetup{colorlinks=true,  linkcolor=blue, citecolor=magenta, urlcolor=cyan}

\makeatletter
\@namedef{subjclassname@2020}{\textup{2020} Mathematics Subject Classification}
\makeatother

\newtheorem{theorem}{Theorem}[section]
\newtheorem{lemma}[theorem]{Lemma}

\newtheorem{problem}[theorem]{Problem}
\newcounter{maintheorem}

\theoremstyle{remark}

\theoremstyle{definition}

\newtheorem{example}[theorem]{Example}
\numberwithin{equation}{section}
\makeatother

\newcommand{\R}{\mathbb{R}}

\newcommand{\nn}[1]{{\left\vert\kern-0.25ex\left\vert\kern-0.25ex\left\vert #1
\right\vert\kern-0.25ex\right\vert\kern-0.25ex\right\vert}}

\renewcommand{\leq}{\leqslant}
\renewcommand{\geq}{\geqslant}

\newcounter{smallromans}

{\end{list}}

\begin{document}
\title[]{ Moduli of continuity and absolute continuity: \\ any relation?}

\author[M.~Muratori]{Matteo Muratori}
\address[M.~Muratori]{Politecnico di Milano, Dipartimento di Matematica, Piazza Leonardo da Vinci 32, 20133 Milano, Italy}
\email{matteo.muratori@polimi.it}

\author[J.~Somaglia]{Jacopo Somaglia}
\address[J.~Somaglia]{Politecnico di Milano, Dipartimento di Matematica, Piazza Leonardo da Vinci 32, 20133 Milano, Italy}
\email{jacopo.somaglia@polimi.it}

\keywords{Absolute continuity; modulus of continuity; Cantor function.}

\subjclass[2020]{Primary: 26A15. Secondary: 26A30; 26A46; 26A48}


\begin{abstract}
We construct a monotone, continuous, but not absolutely continuous function whose minimal modulus of continuity is absolutely continuous. In particular, we establish that there is no equivalence between the absolute continuity of a function and the absolute continuity of its modulus of continuity, in contrast with a well-known property of Lipschitz functions.
\end{abstract}
\maketitle


\section{Introduction}

For a continuous real function $f$ defined on a compact interval $[a,b]$, it is well defined its modulus of continuity
\begin{equation}\label{def-mod}\tag{$\mathsf{mod}$}
\omega_f(\delta)=\max\{|f(x)-f(y)|\colon \ x,y\in [a,b] \, , \ |x-y|\leq\delta\} \qquad \forall \delta \in [0,b-a] \, ,
\end{equation}
which is in turn a nondecreasing continuous function. Note that $ \omega_f $ is usually referred to as the \emph{minimal} modulus of continuity; this will be tacitly assumed throughout the paper when we write $ \omega_f $. Of course, it is possible to relax \eqref{def-mod} upon replacing ``='' with ``$\geq$'', in which case one simply says that $\omega_f$ is \emph{a} modulus of continuity for $f$. However, it is plain that the connection between $f$ and $\omega_f$ becomes weaker. For example, one can always construct a modulus of continuity for $f$ which is \emph{concave} (in particular absolutely continuous), see e.g.~\cite[Lemma 11]{Bre}.

\smallskip

It follows directly from definition \eqref{def-mod} that a function $f$ is  Lipschitz on $[a,b]$ if and only its modulus of continuity $\omega_f$ is also Lipschitz. It is natural to ask whether other types of regularity properties are preserved passing from $f$ to its modulus of continuity $\omega_f$ or vice versa. For instance, in \cite{DP} it has been proved that under suitable conditions the modulus of continuity of a piecewise real-analytic function is analytic at $ \delta=0 $. Here we are interested in \emph{absolute continuity}. It is not hard to show that there exists a continuous function, which is not absolutely continuous, whose modulus of continuity is absolutely continuous (see Example \ref{e: nonmonotona} below). However, the question becomes much more subtle if one assumes that the function $f$ is in addition \emph{monotone}. In this regard, it is interesting to notice that the modulus of continuity of the \emph{Cantor function} is the Cantor function itself, see e.g.~\cite[Proposition 3.2]{DMRV} or \cite{D}. Hence, in this particular case, both the functions $f$ and $\omega_f$ are not absolutely continuous. 

\smallskip 

Elaborating on some further key properties of the Cantor function, in the present note we exhibit a monotone, continuous, not absolutely continuous function whose modulus of continuity is absolutely continuous. In other words, we show that, in general, the absolute continuity of $\omega_f$ does not imply the absolute continuity of $f$, not even if $f$ is monotone. We stress that this result is in sharp contrast with the situation in the Lipschitz setting where, as recalled above, the Lipschitz continuity of $f$ is \emph{equivalent} to the Lipschitz continuity of $ \omega_f $ (in fact with the same Lipschitz constant).

\section{Main result and proof}

Our goal is to prove the following result.

\begin{theorem}\label{t: mainthm}
	There exists a nondecreasing continuous function $f\colon [0,7]\to [0,7]$ such that: 
	\begin{enumerate}
		\item $f$ is not absolutely continuous; 
 \item its minimal modulus of continuity $\omega_f$ is absolutely continuous.
	\end{enumerate}
\end{theorem}

In order to prove the above theorem, we need a few technical lemmas. First of all, let us introduce some basic notations. We denote by $f_1\colon[0,1]\to [0,1]$ the standard Cantor function (see \cite{DMRV} for a wide survey on such a topic) and by $I\colon \R\to \R$ the identity map. If $f\colon[a,b]\to \mathbb{R}$ is a continuous function, then its modulus of continuity is denoted by $ \omega_f : [0,b-a] \to \R^+ $ and defined according to \eqref{def-mod}. On the other hand, it is apparent that for a \emph{nondecreasing} continuous function it holds 
$$
\omega_f(\delta)=\max\{|f(x)-f(y)|\colon \ x,y\in [a,b] \, , \ |x-y| = \delta\} \qquad \forall \delta \in [0,b-a] \, ,
$$
a property that we will exploit systematically.

\smallskip 

As recalled above, the Cantor function is the minimal modulus of continuity of itself. Nevertheless, there exists a much more regular (and explicit) function that can be taken as a larger modulus of continuity.

\begin{lemma}\label{l: concave}
	The function $f_2\colon[0,1]\to[0,1]$, defined by $f_2(x)=x^\alpha$ with $\alpha={\frac{\log(2)}{\log(3)}}$, satisfies the following conditions:
	\begin{enumerate}
		\item\label{c1} $f_2(0)=0$ and $f_2(1)=1$;
		\item\label{i: dominataconcava} $f_1(x)\leq f_2(x)$ for every $x\in [0,1]$;
		\item\label{i: migliormodulo} $|f_1(x)-f_1(y)|\leq f_2(|x-y|)$ for every $x,y\in [0,1]$;
		\item\label{i: identity1} $I(x)\leq f_2(x)$ for every $x\in [0,1]$;
		\item\label{i: identity2} $|I(x)-I(y)|\leq f_2(|x-y|)$ for every $x,y\in [0,1]$.
	\end{enumerate}
\end{lemma}
\begin{proof}
Condition \eqref{c1} is trivially true. From \cite[Proposition 10.1]{DMRV} it holds
	\[
	|f_1(x)-f_1(y)|\leq |x-y|^\alpha \qquad \forall x,y \in [0,1] \, ,
	\]
 that is \eqref{i: migliormodulo}. 
	Therefore, condition \eqref{i: dominataconcava} simply follows by taking $y=0$. 

Finally, conditions \eqref{i: identity1} and \eqref{i: identity2} are straightforward as $\alpha<1$ and we are working in the unit interval $[0,1]$.
\end{proof}

Upon suitably rearranging $f_2$, we can construct a function that satisfies (some of) the conditions of Lemma \ref{l: concave} from below. 

\begin{lemma}\label{l: convex}
The function $f_3\colon[0,1]\to[0,1]$, defined by
$$
f_3(x) = 1-f_2(1-x) \qquad \forall x \in [0,1] \, ,
$$
satisfies the following conditions:
	\begin{enumerate}
		\item $f_3(0)=0$ and $f_3(1)=1$;
		\item\label{i: monotoneconvex} $f_3(x)\leq f_1(x)$ for every $x\in [0,1]$;
		\item\label{i: identity3} $f_3(x)\leq I(x)$ for every $x\in [0,1]$.
	\end{enumerate}
\end{lemma}
\begin{proof}
	 From condition \eqref{c1} in Lemma \ref{l: concave}, it is clear that $f_3(0)=0$ and $f_3(1)=1$. By \cite[Proposition 4.12]{DMRV}, we have that the Cantor function enjoys the symmetry property $f_1(x)=1-f_1(1-x)$ for every $x\in [0,1]$, so that \eqref{i: monotoneconvex} is a direct consequence of condition \eqref{i: dominataconcava} in Lemma \ref{l: concave}. Because of the same property enjoyed by $ I $, condition \eqref{i: identity3} follows too.
\end{proof}

It is apparent that both $f_2$ and $f_3$ are increasing functions belonging to $ C^0([0,1]) \cap C^1((0,1))$. In particular, they are both absolutely continuous.

\smallskip 

We are now in a position to exhibit the function $f\colon [0,7]\to[0,7]$ claimed in the statement of Theorem \ref{t: mainthm}, defined as follows:
\[
f(x)=\begin{cases}
	f_3(x) \,\, &\mbox{if } x\in [0,1]\,;\\
	1+f_2(x-1)\,\, &\mbox{if } x\in (1,2]\,;\\
	2+f_1(x-2)\,\, &\mbox{if } x\in (2,3]\,;\\
	3+f_3(x-3)\,\, &\mbox{if } x\in (3,4]\,;\\
	4+f_2(x-4)\,\, &\mbox{if } x\in (4,5]\,;\\
	5+f_2(x-5)\,\, &\mbox{if } x\in (5,6]\,;\\
	6+f_2(x-6)\,\, &\mbox{if } x\in (6,7]\,.
	\end{cases}
\]
Note that $f$ is nondecreasing, continuous but \emph{not} absolutely continuous, since its restriction to the subinterval $[2,3]$ is equal to a translate of the Cantor function. We then define the function $g\colon [0,7]\to [0,7]$ by replacing the Cantor function with the identity map in such a subinterval:
\[
g(x)=\begin{cases}
	f_3(x) \,\, &\mbox{if } x\in [0,1]\,;\\
	1+f_2(x-1)\,\, &\mbox{if } x\in (1,2]\,;\\
	I(x)\,\, &\mbox{if } x\in (2,3]\,;\\
	3+f_3(x-3)\,\, &\mbox{if } x\in (3,4]\,;\\
	4+f_2(x-4)\,\, &\mbox{if } x\in (4,5]\,;\\
	5+f_2(x-5)\,\, &\mbox{if } x\in (5,6]\,;\\
	6+f_2(x-6)\,\, &\mbox{if } x\in (6,7]\,.
\end{cases}
\]
Clearly, $g$ is increasing, continuous and piecewise $C^1([0,7])$, thus absolutely continuous. Our next goal is to show that the modification in $[2,3]$ does not affect the modulus of continuity. For a (approximate) graph of the two functions, see Figure \ref{f: graphofg} below.

\medskip

\begin{lemma}\label{l: samemodulus}
	The above functions $f$ and $g$ have the same modulus of continuity.
\end{lemma}
\begin{proof}
	Since $f$ and $g$ only differ in $(2,3)$, it is enough to prove the following two claims:
	\begin{enumerate}[(i)]
		\item\label{i: modulusoff} Given any $x,y\in [0,7]$ with $x\in [2,3]$, there exist $x_1,y_1\in [0,7]\setminus(2,3)$ such that $|x-y|=|x_1-y_1|$ and $|f(x)-f(y)|\leq |f(x_1)-f(y_1)|$.
		\item\label{i: modulusofg} Given any $x,y\in [0,7]$ with $x\in [2,3]$, there exist $x_2,y_2\in [0,7]\setminus(2,3)$ such that $|x-y|=|x_2-y_2|$ and $|g(x)-g(y)|\leq |g(x_2)-g(y_2)|$.
	\end{enumerate}
	Let us start from \eqref{i: modulusoff} and fix $x \in [2,3]$. We consider the following seven cases according to the subinterval where $y$ falls, making use of Lemma \ref{l: concave} \eqref{i: dominataconcava}-\eqref{i: migliormodulo}  and Lemma \ref{l: convex} \eqref{i: monotoneconvex} several times:
	\begin{itemize}
		\item if $y\in [0,1]$, we can choose $x_1=x+3$ and $y_1=y+3$;
		\item if $y\in [1,2]$, we can choose $x_1=x-1$ and $y_1=y-1$;
		\item if $y\in [2,3]$, we can choose $x_1=1$ and $ y_1=1+|x-y| $;
		\item if $y\in [3,4]$, we can choose $x_1=x-2$ and $y_1=y-2$;
		\item if $y\in [4,5]$, we can choose $x_1=x+1$ and $y_1=y+1$;
		\item if $y\in [5,6]$, we can choose $x_1=x+1$ and $y_1=y+1$;
		\item if $y\in [6,7]$, we can choose $x_1=x-2$ and $y_1=y-2$. 
	\end{itemize}
	We have therefore proved claim \eqref{i: modulusoff} in all the cases. We finally observe that \eqref{i: modulusofg} follows in the same way as \eqref{i: modulusoff} upon using Lemma \ref{l: concave} \eqref{i: identity1}-\eqref{i: identity2} and Lemma \ref{l: convex} \eqref{i: identity3} in the place of Lemma \ref{l: concave} \eqref{i: dominataconcava}-\eqref{i: migliormodulo} and, respectively, Lemma \ref{l: convex} \eqref{i: monotoneconvex}.
\end{proof}

\begin{figure}[htbp]
\centering
\subfigure
{\includegraphics[scale=0.75]{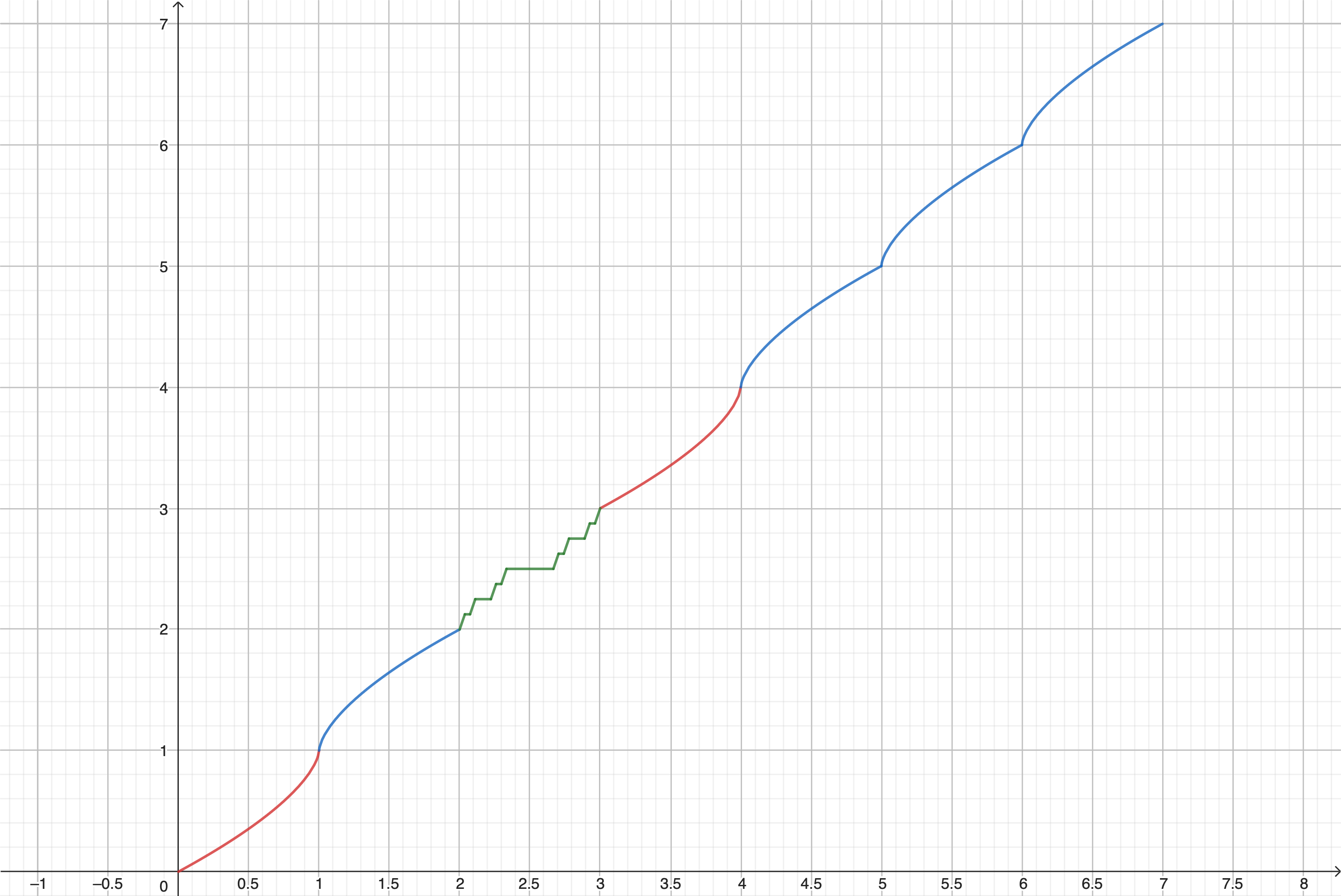}}
\hspace{5mm}
\subfigure
{\includegraphics[scale =0.75]{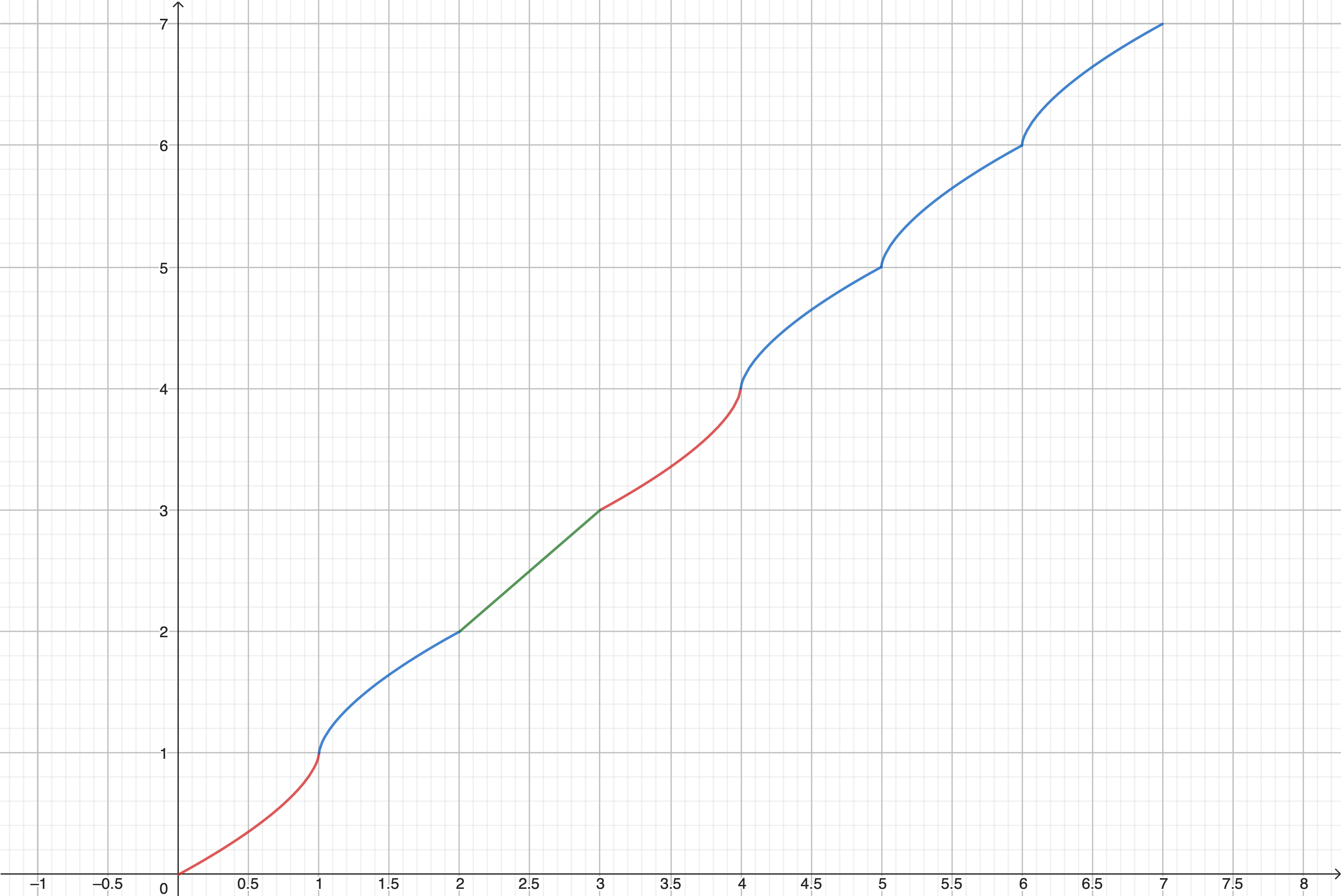}}
\caption{Approximate graphs of the functions $f$ (left) and $g$ (right).} \label{f: graphofg}
\end{figure} 

Hence, we can now concentrate on the regularity of $\omega_g$ (see its graph in Figure \ref{f: graphofomega} below), which is the last step we need to prove our main result.

\begin{lemma}\label{l: modulusofg}
	The function $\omega_g$ is absolutely continuous.
\end{lemma}
\begin{proof}
In order to compute the value $\omega_g(\delta)$ for every fixed $\delta\in (0,7)$ (clearly $ \omega_g(0)=0 $ and $ \omega_g(7)=7 $), we have to maximize the function $\varphi(x)=g(x+\delta)-g(x)$ over $x\in[0,7-\delta]$. Hence, we need to find all $x\in (0,7-\delta)$ such that $\varphi$ is derivable at $x$ and $g'(x+\delta)=g'(x)$, then compare the corresponding values of $\varphi$ with those associated to the points where $ \varphi $ is non-derivable or to the endpoints of the interval. We call the latter set $ A_\delta $, namely 
$$
A_\delta =  \{0,1,2,3,4,5,6\} \cup \{ (1-\delta)^+,  (2 -\delta)^+ ,  (3- \delta)^+,  (4-\delta)^+,  (5-\delta)^+,  (6-\delta)^+,  7-\delta  \}  \, .
$$
Using the basic symmetry and concavity properties of the functions $f_2$ and $f_3$, along with Lemma \ref{l: concave}, it is not difficult to check that
$$
\max_{x\in A_\delta} \varphi(x) = \lfloor \delta \rfloor + f_2(\delta - \lfloor \delta \rfloor) \, .
$$
Let us deal with the critical points of $\varphi$. Thanks to Lemma \ref{l: samemodulus}, among such points, we can ignore those for which either $ x \in (2,3) $ or $ x+\delta \in (2,3) $. Therefore, we must find every $ x \in (0,2) \cup (4,7) \setminus A_\delta $ such that also $ x + \delta \in (0,2) \cup (4,7) \setminus A_\delta $ and 
$$
g'(x+\delta)=g'(x) \, .
$$
Let us denote this (reduced) critical set by $ C_\delta $. Owing to the fact that $ f_2' $ is strictly decreasing, and recalling that $ f_3(x)=1-f_2(1-x) $, we obtain:
$$
C_\delta =
\begin{cases}
\left\{ 1-\frac \delta 2 \, , 4-\frac \delta 2  \right\}  \,   & \text{if } \delta \in (0,1) \, ; \\
\left\{ 1-\frac \delta 2 \, ,  4-\frac \delta 2  \right\} \cup (4,5) \cup (5,6) \,   & \text{if } \delta = 1 \, ; 
\\
\left\{ 1-\frac \delta 2 \, , 2-\frac {\delta-1} 2 \, ,  4-\frac \delta 2 \, , 4-\frac {\delta-1} 2  \right\}  \,   & \text{if } \delta \in (1,2) \, ; 
\\
\left\{ 2-\frac {\delta-1} 2 \, , 4-\frac {\delta-1} 2  \right\}  \cup (4,5) \,   & \text{if } \delta = 2 \, ; 
\\
\left\{ 2-\frac {\delta-1} 2 \, , 4-\frac {\delta-1} 2 \, , 4 - \frac{\delta-2}{2} \right\} \,   & \text{if } \delta \in (2,3) \, ; 
\\
\left\{4-\frac {\delta-2} 2  \right\}  \cup (0,1) \cup (1,2) \,   & \text{if } \delta = 3 \, ; 
\\
\left\{ 1 - \frac{\delta-3}2 \, ,  4-\frac {\delta-2} 2  \right\}  \,   & \text{if } \delta \in (3,4) \, ; 
\\
\left\{ 1 - \frac{\delta-3}2 \right\} \cup (1,2) \,   & \text{if } \delta = 4 \, ; 
\\
\left\{ 1 - \frac{\delta-3}2 \, , 1 - \frac{\delta-4}2  \right\} \,   & \text{if } \delta \in (4,5) \, ; 
\\
\left\{ 1 - \frac{\delta-4}2 \right\} \cup (1,2) \,   & \text{if } \delta =5 \, ; 
\\
\left\{ 1 - \frac{\delta-4}2 \, ,  1 - \frac{\delta-5}2  \right\}  \,   & \text{if } \delta \in (5,6) \, ; 
\\
\left\{ 1 - \frac{\delta-5}2 \right\} \,   & \text{if } \delta \in [6,7) \, .
\end{cases}
$$
Accordingly, it is readily seen that
$$
\max_{x\in C_\delta} \varphi(x) =
\begin{cases}
2 \, f_2\!\left(\frac \delta 2\right)  \,   & \text{if } \delta \in (0,1] \, ; 
\\
\max \left\{ 2 \, f_2\!\left(\frac \delta 2\right) \! , 1+ 2 \, f_2\!\left(\frac {\delta-1} 2\right) \right\}  \,   & \text{if } \delta \in (1,2] \, ; 
\\
\max \left\{ 1+ 2 \, f_2\!\left(\frac {\delta-1} 2\right) \! , 2+ 2 \, f_2\!\left(\frac {\delta-2} 2\right) \right\}  \,   & \text{if } \delta \in (2,3] \, ; 
\\
\max \left\{ 3+ 2 \, f_2\!\left(\frac {\delta-3} 2\right) \! , 2+ 2 \, f_2\!\left(\frac {\delta-2} 2\right) \right\}  \,   & \text{if } \delta \in (3,4] \, ; 
\\
\max \left\{ 3+ 2 \, f_2\!\left(\frac {\delta-3} 2\right) \! , 4+ 2 \, f_2\!\left(\frac {\delta-4} 2\right) \right\}  \,   & \text{if } \delta \in (4,5] \, ; 
\\
\max \left\{ 4+ 2 \, f_2\!\left(\frac {\delta-4} 2\right) \! , 5+ 2 \, f_2\!\left(\frac {\delta-5} 2\right) \right\}  \,  & \text{if } \delta \in (5,6] \, ; 
\\
5+ 2 \, f_2\!\left(\frac {\delta-5} 2\right) \,   & \text{if } \delta \in (6,7) \, .
\end{cases}
$$
In order to determine which is the largest value between those appearing in the above maxima, it is convenient to introduce the auxiliary function 
$$
\psi(\delta) = 2\,f_2\!\left(\tfrac{\delta}{2}\right)-1-2\,f_2\!\left(\tfrac{\delta-1}{2}\right) \qquad \forall \delta \in [1,2]
$$
and observe that, due to the strict concavity of $f_2$, there exists a unique $ \delta^* \in (0,1) $ (not explicit) such that
$$
\psi(\delta)>0 \quad \forall \delta \in [1,1+\delta^*) \qquad \text{and} \qquad \psi(\delta)<0 \quad \forall \delta \in (1+\delta^*,2] \, .
$$
Still by concavity, an elementary comparison reveals that $ \max_{A_\delta} \varphi  \leq \max_{C_\delta} \varphi $. As a result, we finally deduce that
\[
\omega_g(\delta)=\begin{cases}
2\,f_2\!\left(\frac{\delta}{2}\right) \, &\mbox{if } \delta\in [0,1+\delta^*] \, ; \\
1+2\,f_2\!\left (\frac{\delta-1}{2}\right)\, &\mbox{if } \delta\in (1+\delta^*,2+\delta^*] \, ;\\
2+2\,f_2\!\left(\frac{\delta-2}{2}\right) \, &\mbox{if } \delta\in (2+\delta^*,3+\delta^*] \, ;\\
3+2\,f_2\!\left(\frac{\delta-3}{2}\right)\, &\mbox{if } \delta\in (3+\delta^*,4+\delta^*]\,;\\
4+2\,f_2\!\left(\frac{\delta-4}{2}\right)\, &\mbox{if } \delta\in (4+\delta^*,5+\delta^*]\,;\\
5+2\,f_2\!\left(\frac{\delta-5}{2}\right) \, &\mbox{if } \delta\in (5+\delta^*,7] \, ,
\end{cases}
\]
which shows that $ \omega_g $ is piecewise $C^1$ and, therefore, absolutely continuous.

\end{proof}
\begin{figure}[htbp]
\includegraphics[scale=1]{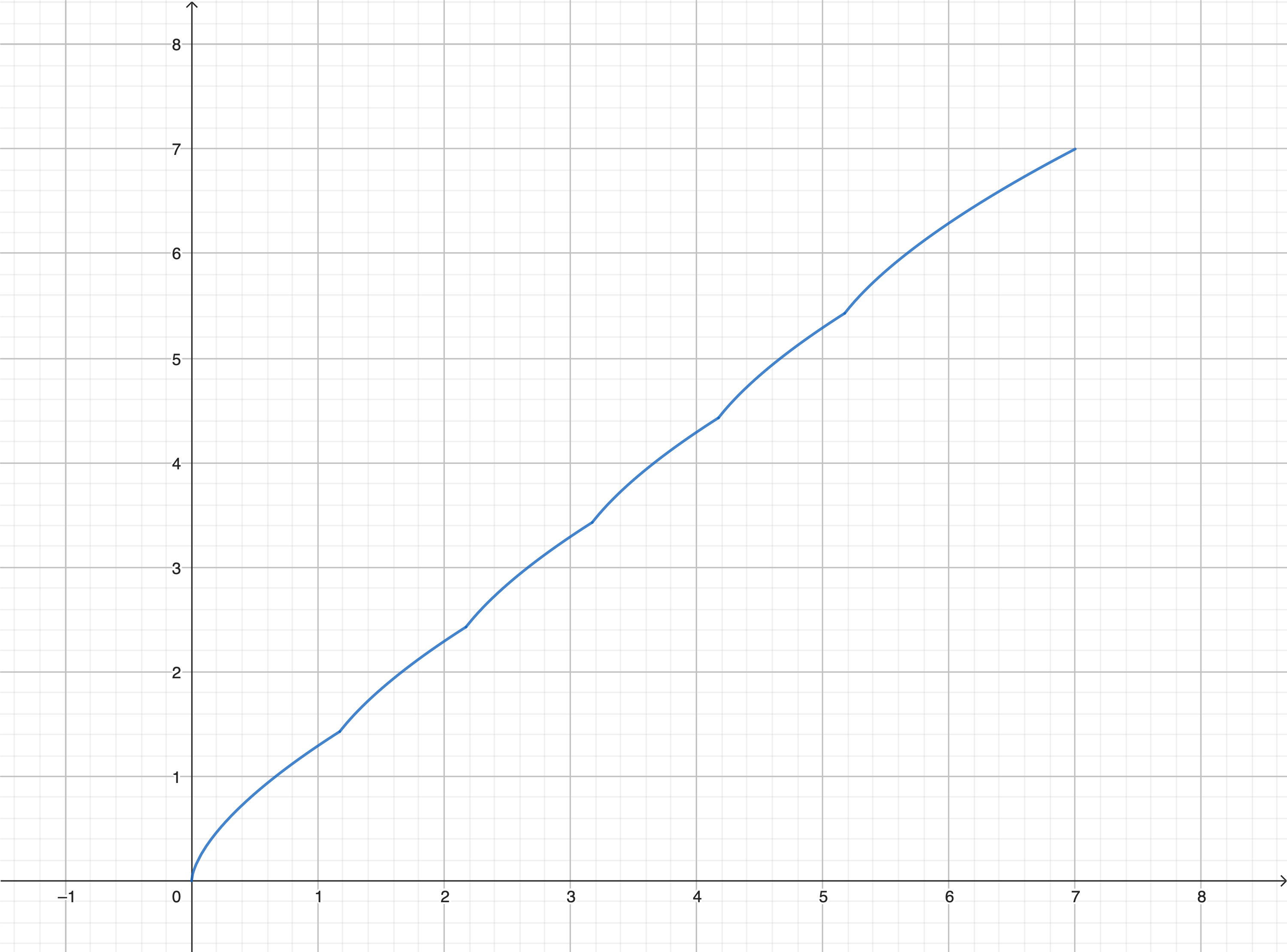}
	\caption{Approximate graph of $\omega_g = \omega_f $.}\label{f: graphofomega}
\end{figure}

\begin{proof}[Proof of Theorem \ref{t: mainthm}]
	It is an immediate consequence of Lemma \ref{l: samemodulus} and Lemma \ref{l: modulusofg}.
\end{proof}

\begin{example}\label{e: nonmonotona}
	The function $h\colon[0,2]\to[0,2]$, defined by
	\[
	h(x)=\begin{cases}
		f_2(x)\, \, & \mbox{if }  x \in [0,1] \, ;\\
		f_1(2-x) \, \, &\mbox{if } x \in (1,2] \, ,
		\end{cases}
	\]
is an elementary example of a non-monotone, continuous, but not absolutely continuous function whose modulus of continuity is absolutely continuous (see Figure \ref{f: graphofh}). Indeed, by Lemma \ref{l: concave} \eqref{i: migliormodulo}, it is readily seen that 
\[
\omega_h(\delta)=\begin{cases}
		f_2(\delta)\,\, &\mbox{if } \delta \in [0,1] \, ;\\
	1 \, \, &\mbox{if }  \delta \in (1,2] \, ,
	\end{cases}
\]
which is absolutely continuous.
\end{example}
	\begin{figure}[htbp]
	\includegraphics[scale=4]{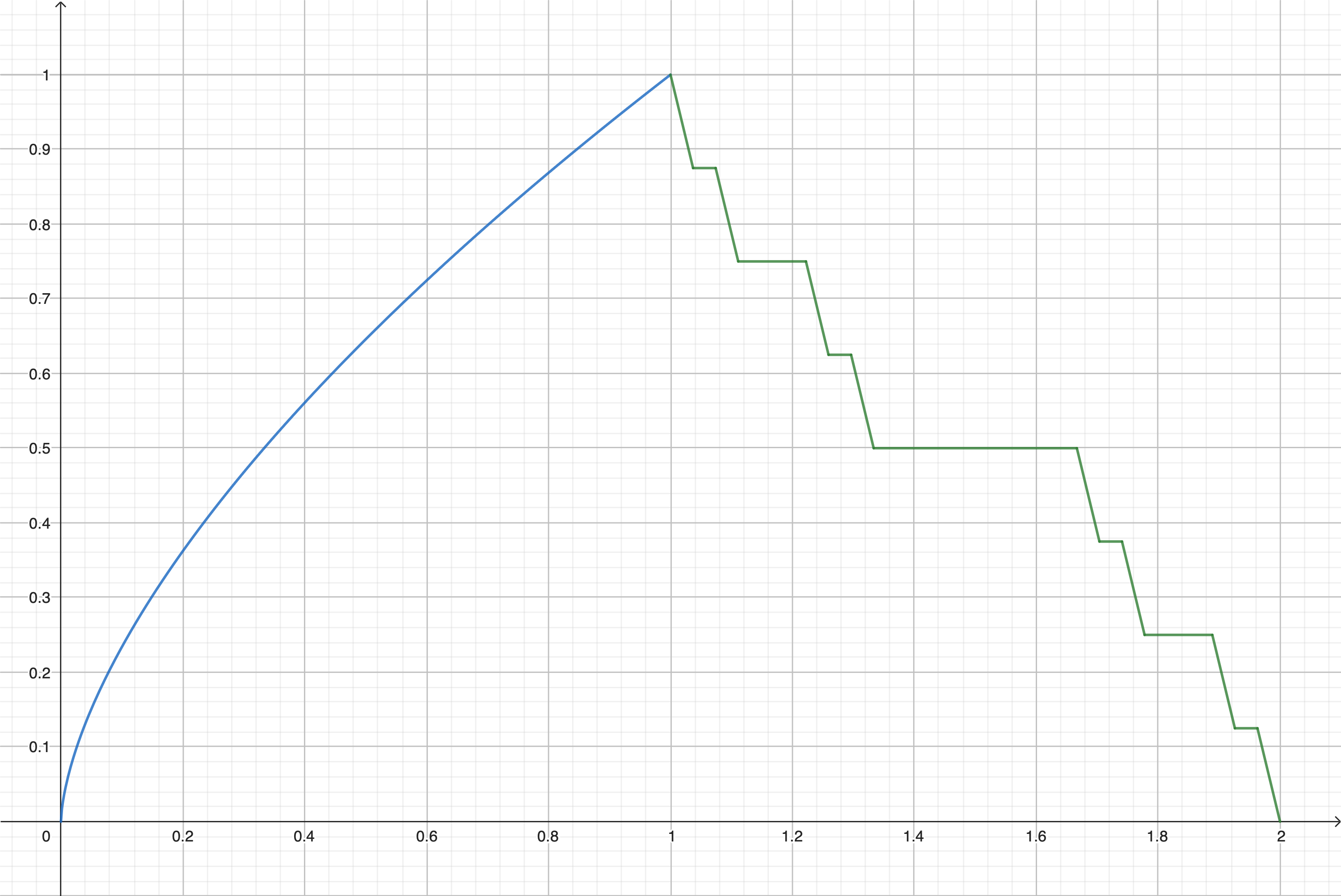}
	\caption{Approximate graph of the function $h$.}\label{f: graphofh}
\end{figure}

We conclude by stating a natural problem that, to the best of our knowledge, seems to be open.

\begin{problem}
	Let $f\colon [0,1]\to[0,1]$ be a (possibly nondecreasing) absolutely continuous function. Is $\omega_f$ absolutely continuous?
\end{problem}

\smallskip  

\section*{Statements \& Declarations}

\textbf{Funding.} The research of the authors has been partially supported by the GNAMPA group (INdAM -- Istituto Nazionale di Alta Matematica, Italy).

\smallskip

\textbf{Competing interests.} The authors have no relevant financial or non-financial interests to disclose.

\smallskip

\textbf{Data availability.} Not applicable.


\end{document}